\documentclass[11 pt]{amsart}
\usepackage{latexsym,amscd,amssymb, graphicx, amsthm, chessfss}  
\usepackage{young}
\usepackage{tikz}

\usepackage[margin=1in]{geometry}

\numberwithin{equation}{section}

\newtheorem{theorem}{Theorem}[section]
\newtheorem{proposition}[theorem]{Proposition}

\newtheorem{lemma}[theorem]{Lemma}

\theoremstyle{definition}

\newcommand{\maj}{{\mathrm {maj}}}
\newcommand{\inv}{{\mathrm {inv}}}
\newcommand{\minimaj}{{\mathrm {minimaj}}}

\newcommand{\dinv}{{\mathrm {dinv}}}

\newcommand{\Des}{{\mathrm {Des}}}
\newcommand{\Val}{{\mathrm {Val}}}
\newcommand{\Rise}{{\mathrm {Rise}}}

\newcommand{\grFrob}{{\mathrm {grFrob}}}

\newcommand{\rev}{{\mathrm {rev}}}
\newcommand{\Hom}{{\mathrm {Hom}}}

\newcommand{\Frob}{{\mathrm {Frob}}}

\newcommand{\rword}{{\mathrm {rword}}}

\newcommand{\iDes}{{\mathrm {iDes}}}

\newcommand{\symm}{{\mathfrak{S}}}

\newcommand{\PP}{{\mathbb {P}}}
\newcommand{\CC}{{\mathbb {C}}}
\newcommand{\QQ}{{\mathbb {Q}}}
\newcommand{\ZZ}{{\mathbb {Z}}}
\newcommand{\NN}{{\mathbb{N}}}
\newcommand{\RR}{{\mathbb{R}}}
\newcommand{\OP}{{\mathcal{OP}}}

\newcommand{\xx}{{\mathbf {x}}}

\newcommand{\yy}{{\mathbf {y}}}

\begin{document}

\title[Hall-Littlewood expansions of Schur delta operators at $t = 0$]
{Hall-Littlewood expansions of Schur delta operators at $t = 0$}

\author{James Haglund}
\address
{Department of Mathematics \newline \indent
University of Pennsylvania \newline \indent
Philadelphia, PA, 19104-6395, USA}
\email{jhaglund@math.upenn.edu}

\author{Brendon Rhoades}
\address
{Department of Mathematics \newline \indent
University of California, San Diego \newline \indent
La Jolla, CA, 92093-0112, USA}
\email{bprhoades@math.ucsd.edu}

\author{Mark Shimozono}
\address
{Department of Mathematics \newline \indent
460 McBryde Hall, Virginia Tech \newline \indent
255 Stanger St. \newline \indent
Blacksburg, VA, 24601, USA}
\email{mshimo@math.vt.edu}

\begin{abstract}
For any Schur function $s_{\nu}$, the associated {\em delta operator}
$\Delta'_{s_{\nu}}$ is a linear operator on the ring of symmetric functions
which has the modified Macdonald polynomials as an eigenbasis.
When $\nu = (1^{n-1})$ is a column of length $n-1$, the symmetric function
$\Delta'_{e_{n-1}} e_n$ appears in the Shuffle Theorem of Carlsson-Mellit.  
More generally,  when 
$\nu = (1^{k-1})$ is any column the polynomial $\Delta'_{e_{k-1}} e_n$ 
is the symmetric function side of the Delta Conjecture of
Haglund-Remmel-Wilson.
We give an expansion of $\omega \Delta'_{s_{\nu}} e_n$ at $t = 0$ 
in the dual Hall-Littlewood basis for any partition $\nu$.
The Delta Conjecture at $t = 0$ was recently proven by
Garsia-Haglund-Remmel-Yoo; our methods give a new proof of this result.
We give an algebraic interpretation of 
$\omega \Delta'_{s_{\nu}} e_n$ at $t = 0$ in terms of a $\mathrm{Hom}$-space.
\end{abstract}

\keywords{Hall-Littlewood function, Macdonald polynomial, delta operator}
\maketitle

\section{Introduction and Main Results}
\label{Introduction}

Let $\Lambda = \bigoplus_{n \geq 0} \Lambda_n$ be the ring of symmetric functions
over the ground field $\QQ(q,t)$ in an infinite variable set $\xx = (x_1, x_2, \dots )$.
Given a partition $\mu$, let $\widetilde{H}_{\mu} = \widetilde{H}_{\mu}(\xx;q,t)$
be the associated {\em modified Macdonald symmetric function}.
The collection $\{ \widetilde{H}_{\mu} \,:\, \text{$\mu$ a partition} \}$ forms a basis
for the ring $\Lambda$.

If $f \in \Lambda$ is any symmetric function, the {\em (unprimed) delta operator}
$\Delta_f: \Lambda \rightarrow \Lambda$ is the Macdonald eigenoperator given by
\begin{equation}
\Delta_f: \widetilde{H}_{\mu} \mapsto f( \dots, q^{i-1} t^{j-1}, \dots) \cdot \widetilde{H}_{\mu},
\end{equation}
where $(i,j)$ ranges over all coordinates in the (English) Ferrers diagram
of the partition $\mu$ (and all remaining variables in $f$ are set to zero).  
As an example, if $\mu = (3,2)$, we fill the Ferrers diagram of $\mu$ with monomials as 
\begin{center}
\begin{Young}
$1$ & $q$ & $q^2$ \cr
$t$ & $qt$
\end{Young}
\end{center}
so that $\Delta_f: \widetilde{H}_{(3,2)} \mapsto f(1, q, q^2, t, qt) \cdot \widetilde{H}_{(3,2)}$.

In this paper, we will focus on a primed version $\Delta'_f: \Lambda \rightarrow \Lambda$
 of the delta operator defined by 
\begin{equation}
\Delta'_f: \widetilde{H}_{\mu} \mapsto f( \dots, q^{i-1} t^{j-1}, \dots) \cdot \widetilde{H}_{\mu},
\end{equation}
where $(i,j)$ range over all coordinates $\neq (0,0)$ in the Ferrers diagram of $\mu$.
If $\mu = (3,2)$ as above, we fill the Ferrers diagram of $\mu$ with monomials as 
\begin{center}
\begin{Young}
$\cdot$ & $q$ & $q^2$ \cr
$t$ & $qt$
\end{Young}
\end{center}
so that $\Delta'_f: \widetilde{H}_{(3,2)} \mapsto f(q, q^2, t, qt) \cdot \widetilde{H}_{(3,2)}$.

Let $k \leq n$ be positive integers.
The {\em Delta Conjecture} of Haglund, Remmel, and Wilson \cite{HRW}
predicts the monomial
expansion of $\Delta'_{e_{k-1}} e_n$ in terms of lattice paths.
It reads 
\begin{equation}
\Delta'_{e_{k-1}} e_n = \Rise_{n,k-1}(\xx;q,t) = \Val_{n,k-1}(\xx;q,t),
\end{equation}
where $\Rise_{n,k-1}(\xx;q,t)$ and $\Val_{n,k-1}(\xx;q,t)$ are certain combinatorially 
defined quasisymmetric functions; see \cite{HRW} for their definitions.

Various special cases of the Delta Conjecture have been proven already.  When 
$k = n$, the Delta Conjecture reduces to the Shuffle Theorem of Carlsson and Mellit \cite{CM}.
In the specialization $q = 1$, Romero \cite{Romero} has proven
\begin{equation}
\Delta'_{e_{k-1}} e_n|_{q = 1} =
\Delta'_{e_{k-1}} e_n|_{t = 1, q = t} =
\Rise_{n,k}(\xx;1,t).
\end{equation}
Zabrocki \cite{Zabrocki} has given evidence for the Delta Conjecture at $t = 1/q$ by showing
that both sides coincide upon pairing with $e_n$ under the Hall inner product.
At $q = 0$, the following theorem summarizes work of Wilson and Rhoades.

\begin{theorem}
\label{delta-zero-start}  (Wilson \cite{WMultiset}, R. \cite{Rhoades})
Let $k \leq n$ be positive integers.  We have
\begin{equation}
\Rise_{n,k-1}(\xx;q,0) = \Rise_{n,k-1}(\xx;0,q) = \Val_{n,k-1}(\xx;q,0) = \Val_{n,k-1}(\xx;0,q).
\end{equation}
\end{theorem}

Theorem~\ref{delta-zero-start} is proven by interpreting the four formal power series
therein in terms of four statistics 
(called $\inv, \maj, \dinv,$ and $\minimaj$) on ordered multiset partitions, and then
proving the relevant equidistribution results.  
Let $C_{n,k} = C_{n,k}(\xx;q)$ be the common symmetric function of 
Theorem~\ref{delta-zero-start}:
\begin{equation}
C_{n,k}  := 
\Rise_{n,k-1}(\xx;q,0) = \Rise_{n,k-1}(\xx;0,q) = 
\Val_{n,k-1}(\xx;q,0) = \Val_{n,k-1}(\xx;0,q).
\end{equation}

The authors of this paper showed \cite{HRS} that the image
$\omega C_{n,k}$ of $C_{n,k}$ under the $\omega$ involution 
has the following expansion
in the dual Hall-Littlewood basis:
\begin{equation}
\label{c-hl-expansion}
\omega C_{n,k} = \sum_{\substack{\mu \vdash n \\ \ell(\mu) = k}}
q^{\overline{b}(\mu)}{k \brack m(\mu)}_q \cdot Q'_{\mu}.
\end{equation}
Here ${k \brack m(\mu)}_q$ is the $q$-multinomial coefficient corresponding
to the part multiplicities of $\mu$, the numbers $b(\mu)$ and $\overline{b}(\mu)$ are
given by
\begin{equation}
\begin{cases}
b(\mu) = \sum_i \mu_i(i-1) \\
\overline{b}(\mu) = \sum_i (\mu_i - 1)(i-1),
\end{cases}
\end{equation}
and 
$Q'_{\mu} = Q'_{\mu}(\xx;q)$ is the dual Hall-Littlewood symmetric function related 
to the Schur basis by
\begin{equation}
\label{Q-to-s}
Q'_{\mu} = \sum_{\lambda} K_{\lambda,\mu}(q) s_{\lambda},
\end{equation}
where $K_{\lambda,\mu}(q) \in \ZZ_{\geq 0}[q]$ is the Kostka-Foulkes polynomial.

Garsia, Haglund, Remmel, and Yoo \cite{GHRY}
recently proved the Delta Conjecture at $t = 0$
by using plethystic methods and Equation~\eqref{c-hl-expansion} to show
\begin{equation}
\label{delta-zero-equation}
\Delta'_{e_{k-1}} e_n |_{t = 0} = \Delta'_{e_{k-1}} e_n|_{q = 0, t = q} = C_{n,k}.
\end{equation} 
We give a new proof of 
Equation~\eqref{delta-zero-equation} using skewing operators $e_j^{\perp}$ 
on the ring $\Lambda$ of symmetric functions together with 
${}_3 \phi_2$-hypergeometric transformations.

Finding positive $Q'$-basis expansions of symmetric functions is interesting 
for several reasons.  Equation~\eqref{Q-to-s} shows that any symmetric 
function with a positive $Q'$ expansion is automatically Schur positive, and thus
is the Frobenius image some module over the symmetric group $\symm_n$.
Even better, the function $Q'_{\mu}$ is itself (up to a twist) the Frobenius image
of the action of $\symm_n$ on the cohomology of the {\em Springer fiber}
$\mathcal{B}_{\mu}$ or on the quotient of the polynomial ring
$\QQ[x_1, \dots, x_n]$ by the {\em Tanisaki ideal} $I_{\mu}$.
We generalize 
Equation~\eqref{delta-zero-equation}
to find the $Q'$-basis expansion of
$\omega \Delta'_{s_{\nu}} e_n |_{t = 0}$ for any 
partition $\nu$.

\begin{theorem}
\label{delta-hl-expansion-theorem}
Let $\nu$ be a partition and let $n \geq 0$.  We have
\begin{equation}
\omega \Delta'_{s_{\nu}} e_n |_{t = 0} =  \sum_{k = \ell(\nu) + 1}^{|\nu| + 1}
P_{\nu, k - 1}(q)
\sum_{\substack{\mu \vdash n \\ \ell(\mu) = k}}
q^{\overline{b}(\mu)} 
\cdot {k \brack m(\mu)}_q 
\cdot Q'_{\mu},
\end{equation}
where
\begin{equation}
P_{\nu,k-1}(q) = q^{|\nu| - {k  \choose 2}} 
 \sum_{\substack{|\rho| = |\nu| \\ \ell(\rho) = k-1}} q^{b(\rho)}  {k-1 \brack m(\rho)}_q K_{\nu,\rho}(q)
\end{equation}
and $K_{\nu,\rho}(q)$ is the Kostka-Foulkes polynomial.
\end{theorem}

As operators on $\Lambda$ we have the identity
\begin{equation}
\Delta_{s_{\nu}} = \sum_{\rho \subseteq \nu} \Delta'_{s_{\rho}},
\end{equation}
where $\rho$ ranges over all partitions obtainable from $\nu$ by removing a horizontal
strip.  Theorem~\ref{delta-hl-expansion-theorem} therefore also gives a positive expansion 
for  $\omega \Delta_{s_{\nu}} e_n |_{t = 0}$ in the $Q'$-basis, where we are using an unprimed
delta operator.

Haiman proved that the symmetric function $\Delta'_{e_{n-1}} e_n = \Delta_{e_n} e_n$ 
(otherwise known as $\nabla e_n$) is the bigraded Frobenius image of the diagonal
coinvariant ring \cite{Haiman}.  It is an open problem to give (even conjecturally)
a corresponding algebraic interpretation of the symmetric function
$\Delta'_{e_{k-1}} e_n$ appearing in the Delta Conjecture.

If $\nu$ with $\ell(\nu) = n$, Haiman \cite{Haiman} gave an algebraic interpretation of 
$\Delta_{s_{\nu}} e_n$ as a Schur functor applied to a vector bundle 
over the Hilbert scheme of $n$ points in the plane $\CC^2$.  In particular,
Haiman's result implies that $\Delta_{s_{\nu}} e_n$ is Schur positive when
$\ell(\nu) = n$.  Haiman conjectured that $\Delta_{s_{\nu}} e_n$ is Schur positive
for any partition $\nu$.  Haglund and Wilson have computational evidence
that $\Delta'_{s_{\nu}} e_n$ is also Schur positive for any partition $\nu$.
Theorem~\ref{delta-hl-expansion-theorem} gives evidence for the Schur positivity
of $\Delta'_{s_{\nu}} e_n$ (and thus also $\Delta_{s_{\nu}} e_n$) for 
arbitrary partitions $\nu$.

In \cite{HRS} the authors found an algebraic interpretation of the Delta Conjecture
at $t = 0$.  Let the symmetric group $\symm_n$ act on the polynomial ring
$\QQ[x_1, \dots, x_n]$ in $n$ variables.  Following \cite[Defn. 1.1]{HRS}, given
positive integers $k \leq n$ we define the ideal $I_{n,k} \subseteq \QQ[x_1, \dots, x_n]$ by
\begin{equation}
I_{n,k} := \langle e_n, e_{n-1}, \dots, e_{n-k+1}, x_1^k, x_2^k, \dots, x_n^k \rangle
\end{equation}
and let 
\begin{equation}
R_{n,k} := \QQ[x_1, \dots, x_n]/I_{n,k}
\end{equation}
be the corresponding quotient.
When $k = n$ the ring $R_{n,k}$ reduces to the classical {\em coinvariant algebra}
$R_n = \QQ[x_1, \dots, x_n] / \langle e_1, \dots, e_n \rangle$ 
obtained by modding out by symmetric polynomials in $\QQ[x_1, \dots, x_n]$ with vanishing
constant term.
Just as algebraic properties of $R_n$ are governed by combinatorial properties of 
permutations in $\symm_n$, it is shown in \cite{HRS} that algebraic properties of 
$R_{n,k}$ are governed by {\em ordered set partitions} of
$[n] := \{1, 2, \dots, n\}$ with $k$ blocks.

The ring $R_{n,k}$ has the structure of a graded $\symm_n$-module; in 
\cite{HRS} it is proven that its graded Frobenius image is
\begin{equation}
\label{r-frobenius}
\grFrob(R_{n,k};q) = (\rev_q \circ \omega) C_{n,k},
\end{equation}
where $\rev_q$ is the operator which reverses the coefficient sequences of polynomials
in $q$, e.g. 
\begin{equation*}
\rev_q(3 s_{(2,1)} q^2 + 2 s_{(1,1,1)} q + s_{(3)}) =  s_{(3)} q^2 + 2 s_{(1,1,1)} q + 3 s_{(2,1)}.
\end{equation*}
Thanks to the Garsia-Haglund-Remmel-Yoo
Equation~\eqref{delta-zero-equation} we can also express 
Equation~\eqref{r-frobenius} as
\begin{equation}
\label{r-frobenius-delta}
\grFrob(R_{n,k};q) = (\rev_q \circ \omega) \Delta'_{e_{k-1}} e_n |_{t = 0}.
\end{equation}
Informally, we think of $R_{n,k}$ as the `coinvariant algebra' attached to the operator
$\Delta'_{e_{k-1}} e_n$.

Given Equation~\eqref{r-frobenius-delta}, one could ask for a graded $\symm_n$-module
$R_{\nu,n}$ which satisfies
\begin{equation}
\label{goal-frobenius-r-nu}
\grFrob(R_{\nu,n};q) = 
(\rev_q \circ \omega) \Delta'_{s_{\nu}} e_n |_{t = 0}
\end{equation}
for any partition $\nu$.
This would give a coinvariant algebra attached to the operator $\Delta'_{s_{\nu}}$.
In \cite{RW} Rhoades and Wilson exhibited a quotient of $\QQ[x_1, \dots, x_n]$
with graded Frobenius image 
$(\rev_q \circ \omega) \Delta_{s_{\nu}} e_n |_{t = 0}$ when $\nu$ is a hook of the form
$(r,1^{n-1})$.

For general partitions $\nu \vdash m$, it is impossible to exhibit a module $R_{n,\nu}$ 
satisfying Equation~\eqref{goal-frobenius-r-nu} as a submodule of 
$\QQ[x_1, \dots, x_n]$; the graded components of the polynomial ring
are not large enough for this purpose.
Given Theorem~\ref{delta-hl-expansion-theorem}, two artificial solutions to this problem
are as follows.
\begin{itemize}
\item  Very artificially, we could use the positive expansion of $Q'_{\mu}$ in the Schur basis
$\{ s_{\lambda} \}$ and define $R_{n,\nu}$ as a direct sum of $\symm_n$-irreducibles
$S^{\lambda}$ with appropriate grading shifts.
\item  Less artificially, we could use the fact that $\rev_q(Q'_{\mu})$ is the 
graded Frobenius image of the {\em Tanisaki quotient} 
$R_{\mu} = \QQ[x_1, \dots, x_n]/I_{\mu}$, where $I_{\mu}$ is the 
Tanisaki ideal.  
Theorem~\ref{delta-hl-expansion-theorem} then leads to a definition of 
$R_{n,\nu}$ as a direct sum of $R_{\mu}$'s with appropriate grading shifts.
\end{itemize}
The second bullet point is less artificial because the pieces $R_{\mu}$ which constitute the 
module $R_{n,\nu}$ are larger than the pieces $S^{\lambda}$ appearing in the first 
bullet point.

In this paper we give a still less artificial construction for $R_{n,\nu}$ as a Hom-space.
For any $n, m \geq 0$ we
define a graded $\symm_m \times \symm_n$-module
$V_{n,m}$ by
\begin{equation}
\label{v-definition}
V_{n,m} := \bigoplus_{k \geq 0} (R_{m,k-1} \otimes R_{n,k})
\{ -mn + km + kn - n - k(k-1) \}.
\end{equation}
Here $M \{ -d\}$ denotes a graded module $M$ with degree shifted up by $d$
and we impose grading on tensor products by declaring
\begin{equation}
(M \otimes N)_d = \bigoplus_{i + j = d} M_i \otimes N_j.
\end{equation}
If $M$ is any $\symm_m$-module, the Hom-space 
$\Hom_{\symm_m}(M, V_{n,m})$ has the structure of a graded $\symm_n$-module.

\begin{theorem}
\label{algebraic-interpretation}
Let $n \geq 0$ and let $\nu \vdash m$ be a partition.  Define the graded 
$\symm_n$-module $R_{n,\nu}$ by
\begin{equation}
R_{n,\nu} := \Hom_{\symm_m}(S^{\nu}, V_{n,m}) \{b(\nu)\},
\end{equation}
where $V_{n,m}$ is defined as in Equation~\eqref{v-definition}.  We have
\begin{equation}
\grFrob(R_{n,\nu}; q) = (\rev_q \circ \omega) \Delta'_{s_{\nu}} e_n |_{t = 0}.
\end{equation}
\end{theorem}

The module $V_{n,m}$ only depends on the integers $n$ and $m = |\nu|$, so that
we may regard $V_{n,m}$ as a universal `generator' for coinvariant algebras
corresponding to partitions $\nu \vdash m$.
Since the $R_{n,k}$ modules are `larger' than the Tanisaki quotients $R_{\mu}$,
this gives a still less artificial solution to finding a `coinvariant algebra'
attached to the operator $\Delta'_{s_{\nu}}$.
This also suggests that understanding the algebraic and geometric properties
of objects related to $\Delta'_{s_{\nu}}$ at $t = 0$ may be deduced from the case where 
$\nu$ is a single column.

\section{Background}
\label{Background}

\subsection{Symmetric functions}
We adopt standard symmetric function terminology which may be found in e.g. 
\cite{Hagbook, Macdonald}.
Given a partition $\lambda$, let 
\begin{equation*}
e_{\lambda} = e_{\lambda}(\xx), \quad 
h_{\lambda} = h_{\lambda}(\xx), \quad
s_{\lambda} = s_{\lambda}(\xx), \quad 
Q'_{\lambda} = Q'_{\lambda}(\xx;q), \quad
\widetilde{H}_{\lambda} = \widetilde{H}_{\lambda}(\xx;q,t)
\end{equation*}
be the associated {\em elementary, homogeneous, Schur,}  {\em dual 
Hall-Littlewood}, and {\em modified Macdonald symmetric function}.  
The functions $Q'_{\mu}$ expand positively
in the Schur basis.  If $\mu \vdash n $, the transition coefficients 
$K_{\lambda,\mu}(q) \in \ZZ_{\geq 0}[q]$ given by
\begin{equation}
Q'_{\mu} = \sum_{\lambda \vdash n} K_{\lambda,\mu}(q) s_{\lambda}
\end{equation}
are the {\em Kostka-Foulkes polynomials}.
The following relationship between the modified Macdonald symmetric functions and
the dual Hall-Littlewood functions is well known:
\begin{equation}
\widetilde{H}_{\lambda}|_{t = 0} = \rev_q (Q'_{\lambda}).
\end{equation}

If $S \subseteq [n-1]$, we let $F_{n,S}$ be the associated {\em fundamental quasisymmetric
function} of degree $n$ given by
\begin{equation}
F_{n,S} := \sum_{\substack{i_1 \leq \cdots \leq i_n \\ (j \in S) \Rightarrow (i_j < i_{j+1}}}
x_{i_1} \cdots x_{i_n}.
\end{equation}

We adopt the usual $q$-analogs of numbers, factorials, binomial
coefficients, and multinomial coefficients:
\begin{equation}
\begin{cases}
[n]_q := 1 + q + q^2 + \cdots + q^{n-1} & n \geq 0  \\
[n]!_q := [n]_q [n-1]_q \cdots [1]_q & n \geq 0 \\
{n \brack k}_q := \frac{[n]!_q}{[k]!_q [n-k]!_q}  & n \geq k \geq 0 \\
{n \brack a_1, \dots, a_k}_q  := \frac{[n]!_q}{[a_1]!_q \cdots [a_k]!_q} & a_1 + \cdots + a_k = n.
\end{cases}
\end{equation}

If $\lambda = (\lambda_1, \dots, \lambda_k)$ is a partition, we let $\ell(\lambda) = k$
be the number of parts of $\lambda$, let 
$|\lambda| = \lambda_1 + \cdots + \lambda_k$ be the sum of the parts of $\lambda$, 
and set
\begin{align}
b(\lambda) &:= \sum_{i = 1}^k \lambda_i \cdot (i-1) \\
\overline{b}(\lambda) &:=
b(\lambda) - {\ell(\lambda) \choose 2} = \sum_{i = 1}^k (\lambda_i - 1)(i-1).
\end{align}
We let $m_i(\lambda)$ denote the multiplicity of $i$ as a part of $\lambda$
and adopt the $q$-multinomial coefficient shorthand
\begin{equation}
{\ell(\lambda) \brack m(\lambda}_q := 
{\ell(\lambda) \brack m_1(\lambda), m_2(\lambda), \dots }_q.
\end{equation}

Let $\omega$ be the involution on $\Lambda$ which interchanges $e_n$ and $h_n$.
We will use the following `twisted' version of the polynomials $C_{n,k}$ for $k \leq n$:
\begin{equation}
D_{n,k} := (\rev_q \circ \omega) C_{n,k}.
\end{equation}

Let $\langle \cdot, \cdot \rangle$ be the {\em Hall inner product}
on $\Lambda$ defined by the declaring the Schur functions to be
orthonormal:
$\langle s_{\lambda}, s_{\mu} \rangle = \delta_{\lambda, \mu}$.
For any symmetric function $f$, the operator $f^{\perp}: \Lambda \rightarrow \Lambda$
is the dual operator to multiplication by $f$ under the Hall inner product.
Said differently, the operator $f^{\perp}$ is characterized by 
\begin{equation}
\langle f^{\perp} g, h \rangle = \langle g, fh \rangle,
\end{equation}
for all $g, h \in \Lambda$.
For a proof of the following standard fact, see for example \cite[Lem. 3.6]{HRS}.

\begin{lemma}
\label{e-perp-lemma}
Let $f, g \in \Lambda$ be symmetric functions with equal constant terms.  We have
that $f = g$ if and only if $e_j^{\perp} f = e_j^{\perp} g$ for all $j \geq 1$.
\end{lemma}

Lemma~\ref{e-perp-lemma} will be used to form the recursions which underly our 
new proof of the Delta Conjecture at $t = 0$.
The image of the $D_{n,k}$ functions under $e_j^{\perp}$ can be
recursively described as follows.

\begin{lemma}
\label{d-e-perp-lemma}
(H.-R.-S. \cite[Lem. 3.7]{HRS})
Let $k \leq n$ be positive integers and let $j \geq 1$.  We have
\begin{equation}
e_j^{\perp} D_{n,k} = 
q^{{j \choose 2}} {k \brack j}_q \cdot
\sum_{m = \max(1,k-j)}^{\min(k,n-j)}
q^{(k-m) \cdot (n-j-m)}
{j \brack k-m}_q D_{n-j,m}.
\end{equation}
\end{lemma}

The irreducible representations of the symmetric group $\symm_n$ over $\QQ$
are indexed by partitions $\lambda \vdash n$.  If $\lambda$ is a partition, we let 
$S^{\lambda}$ denote the corresponding irreducible representation.  If $V$ is any
finite-dimensional $\symm_n$-module, there exist unique integers 
$c_{\lambda} \geq 0$ such that 
$V \cong_{\symm_n}  \bigoplus_{\lambda \vdash n} c_{\lambda} S^{\lambda}$.
The {\em Frobenius image} $\Frob(V) \in \Lambda_n$ is the symmetric function
\begin{equation}
\Frob(V) := \sum_{\lambda \vdash n} c_{\lambda} s_{\lambda}.
\end{equation}
More generally, if $V = \bigoplus_{d \geq 0} V_d$ is a graded $\symm_n$-module
with each $V_d$ finite-dimensional, the {\em graded Frobenius image}
$\grFrob(V)$ is 
\begin{equation}
\grFrob(V) = \sum_{d \geq 0} \Frob(V_d) \cdot q^d.
\end{equation}

\subsection{Ordered set partitions}
If $\pi = \pi_1 \pi_2 \cdots \pi_n \in \symm_n$ is a permutation
(written in one-line notation), the {\em descent set} of $\pi$ is
\begin{equation}
\Des(\pi) := \{ 1 \leq i \leq n-1 \,:\, \pi_i > \pi_{i+1} \}
\end{equation}
and the {\em inverse descent set} is $\iDes(\pi) := \Des(\pi^{-1})$.

Let $k \leq n$ be positive integers.  An {\em ordered set partition of size $n$ with $k$ blocks}
is a sequence $\sigma = (B_1 \mid \cdots \mid B_k)$ of $k$ nonempty subsets of 
$[n]$ such that we have a disjoint union decomposition 
$[n] = B_1 \uplus \cdots \uplus B_k$.  
Let $\OP_{n,k}$ be the family of ordered set partitions of size $n$ 
with $k$ blocks. As an example, we have
$\sigma = (27 \mid 135 \mid 46) \in \OP_{7,3}$.  There is a natural identification
$\OP_{n,n} = \symm_n$ of ordered set partitions of size $n$ with $n$ blocks
and permutations in the symmetric group on $n$ letters.

Let $\sigma = (B_1 \mid \cdots \mid B_k) \in \OP_{n,k}$ be an ordered set partition.
An {\em inversion} in $\sigma$ is  a pair $1 \leq i < j \leq n$ such that 
\begin{itemize}
\item
$i$'s block
is strictly to the right of $j$'s block in $\sigma$ and 
\item $i$ is minimal in its block.
\end{itemize}
We let $\inv(\sigma)$ be the number of inversions of $\sigma$.
For example, if $\sigma = (27 \mid 135 \mid 46)$, the inversions of 
$\sigma$ are $12, 17, 47,$ and $45$ so that $\inv(\sigma) = 4$.

If $\sigma = (B_1 \mid \cdots \mid B_k) \in \OP_{n,k}$, the 
{\em reading word} $\rword(\sigma)$ is the permutation 
$\pi_1 \pi_2 \cdots \pi_n \in \symm_n$ obtained by reading 
$\sigma$ along `diagonals' from left to right (where the $m^{th}$
`diagonal' is  the set of elements which are $m^{th}$ largest in their block).
As an example, we have
\begin{equation*}
\rword(27 \mid 135 \mid 45) = 5736214.
\end{equation*}

We have (see \cite{HRW} or \cite[Eq. 2.20]{HRS}) the following quasisymmetric
expansion of $C_{n,k}$ in terms of ordered set partitions:
\begin{equation}
C_{n,k} = \sum_{\sigma \in \OP_{n,k}} q^{\inv(\sigma)} F_{n, \iDes(\rword(\sigma))}.
\end{equation}

\subsection{Hypergeometric functions}
Given a continuous parameter $x$ and an integer $k \geq 0$, the
{\em Pochhammer symbol} is
\begin{equation}
(x)_k = (x;q)_k = (1-x)(1-xq) \cdots (1-xq^{k-1}).
\end{equation}
We adopt the abbreviation
\begin{equation}
\prod_{i = 1}^j (a_i)_k := (a_1, a_2, \dots, a_k)_j.
\end{equation}

If $r, s \geq 0$ are nonnegative integers and $\alpha_1, \dots, \alpha_r$
and $\beta_1, \dots, \beta_s$ are parameters, the 
corresponding {\em $q$-hypergeometric series} is
\begin{equation}
{ _r\phi _s} \left ( \begin{matrix} \alpha_1, \dots, \alpha_r \\ 
\beta_1, \dots, \beta_s \end{matrix} ; q, z \right) :=
\sum_{n = 0}^{\infty} 
\frac{ (\alpha_1; q)_n \cdots (\alpha_r; q)_n}
{ (\beta_1; q)_n \cdots (\beta_s; q)_n}
\frac{z^n}{(q;q)_n} 
\left[ (-1)^n q^{{n \choose 2}} \right]^{1+s-r}.
\end{equation}
In this paper we will only be concerned with the ${_3\phi_2}$-functions.

\section{Polynomial identities}
\label{Polynomial}

In this section we will prove symmetric function and hypergeometric identities
which will be used in our proof of the Delta Conjecture at $t = 0$, and ultimately
in our proof of Theorem~\ref{delta-hl-expansion-theorem}.
The first of these is a recursive description of the image of $C_{n,k}$ under the operator
$e_j^{\perp}$.

\begin{lemma}
\label{c-e-perp-recursion}
Let $j \geq 1$ and $k \leq n$.  We have
\begin{equation}
e_j^{\perp} C_{n,k} =
\sum_{r = 0}^j q^{r \choose 2} {k \brack r}_q {k + j - r - 1 \brack j - r}_q  C_{n-j,k-r}.
\end{equation}
\end{lemma}

The proof of Lemma~\ref{c-e-perp-recursion} should be compared with that of
Lemma~\ref{d-e-perp-lemma} ( = \cite[Lem 3.7]{HRS}).

\begin{proof}
We start with the quasisymmetric
expansion of $C_{n,k}$ in terms of ordered set partitions:
\begin{equation}
C_{n,k} = \sum_{\sigma \in \OP_{n,k}} q^{\inv(\sigma)} F_{n, \iDes(\rword(\sigma))}.
\end{equation}

Let $\alpha = (\alpha_1, \alpha_2, \dots, \alpha_p)$ be any (strict) composition
of $n$.  General facts about superization (see \cite{Hagbook})
imply
\begin{equation}
\langle C_{n,k}, e_{\alpha_1} e_{\alpha_2} \cdots e_{\alpha_p} \rangle =
\sum_{\substack{\sigma \in \OP_{n,k} \\ \rword(\sigma) \text{ is an $\alpha$-shuffle}}}
q^{\inv(\sigma)}.
\end{equation}
The $\alpha$-shuffle condition means that the sequence $\rword(\sigma)$ is a shuffle of the $p$ decreasing sequences
$$
(\alpha_1, \dots, 2, 1), (\alpha_1 + \alpha_2, \dots, \alpha_1 + 2, \alpha_1 + 1), \dots, 
(n, n-1, \dots, \alpha_1 + \cdots + \alpha_{p-1} + 1).
$$
We are interested in the case where $\alpha_1 = j$, so that 
\begin{equation}
\langle C_{n,k}, e_{\alpha_1} e_{\alpha_2} \cdots e_{\alpha_p} \rangle =
\langle C_{n,k}, e_{j} e_{\alpha_2} \cdots e_{\alpha_p} \rangle =
\langle  e_j^{\perp} C_{n,k}, e_{\alpha_2} \cdots e_{\alpha_p} \rangle.
\end{equation}
As in the proof of \cite[Lem. 3.7]{HRS}, we give a combinatorial interpretation of this expression.

Fix an index $0 \leq r \leq k-1$ and consider $\sigma \in \OP_{n-j,k-r}$.  Let $T$ be a way of adding the $j$
letters $\{n-j+1, \dots, n\}$ to $\sigma$ (the {\em big} letters) in such a way that the resulting ordered set partition
$\sigma'$ has $k$ blocks and the big letters appear in $\rword(\sigma')$ in the order 
$n, n-1, \dots, n-j+1$.  
An example of such a way $T$ for $n = 9, k = 5, j = 4, r = 2$ is shown below, with the big letters in bold:
$$
(4 \mid 15 \mid 23) \leadsto
( 4   \mid 1 5 {\bf 9 } \mid {\bf 7} {\bf 8} \mid 2 3  \mid {\bf 6} ).
$$
Notice that exactly $r$ of the big letters to be minimal blocks of $\sigma'$,
and these minimal  letters must be (from left to right)
$$
n - j + r , \dots, n - j + 2 ,n - j + 1 .
$$

Let us consider the effect on the $\inv$ statistic of all possible ways $T$ of producing $\sigma'$ from $\sigma$.
Call the letters $1, 2, \dots, n-j$ of $\sigma'$ which are not big {\em small}.
Following the notation of \cite{HRS}, let us call a letter $i$ of $\sigma'$
\begin{itemize}
\item  $mins$ if $i$ is small and minimal in its block,
\item  $minb$ if $i$ is big and minimal in its block,
\item  $nmins$ if $i$ is small and not minimal in its block, and
\item  $nminb$ if $i$ is big and not minimal in its block.
\end{itemize}
We observe the following.
\begin{itemize}
\item
The $r$ letters $n-j+r, \dots, n-j+1$ are precisely the $minb$ letters for any way $T$, and they contribute amongst 
themselves ${r \choose 2}$ inversions in $\sigma'$.
\item
The ${k \choose r}$ ways of distributing the $minb$ letters among the blocks of $\sigma'$ generate inversions
with the $mins$ letters, contributing to a factor of ${k \brack r}_q$ in the generating function for $\inv(\sigma')$ when we sum
over all ways $T$.
\item  
The values of the $nminb$ letters are completely determined by which block they are added to.
There are $j - r$ letters which are $nminb$ and they may be added to any of the $k$ blocks
of $\sigma$ (upon addition
of the $minb$ letters), with multiplicity.  This gives ${k + j - r - 1 \choose j - r}$ choices for distributing the $nminb$ letters in $T$.
The inversions contributed between the $nminb$ letters and the $mins$ letters generate a factor of
${k + j - r - 1 \brack j - r}_q$ to the generating function for $\inv(\sigma')$ when we sum over all ways $T$.
\end{itemize}

By the last paragraph, we have
\begin{equation}
\sum_{T: \sigma \leadsto \sigma'} q^{\inv(\sigma')} = q^{\inv(\sigma) + {r \choose 2}} {k \brack r}_q {k + j - r - 1 \brack j - r}_q,
\end{equation}
where the sum is over all ways $T$ of producing $\sigma'$ from $\sigma$.  If we sum this expression over all $\sigma$
with $\rword(\sigma)$ an $(\alpha_2, \dots, \alpha_p)$-shuffle, and then over all $r$, we get the inner product
\begin{equation}
\left\langle \sum_{r = 0}^j q^{r \choose 2} {k \brack r}_q {k + j - r - 1 \brack j - r}_q  C_{n-j,k-r}, e_{\alpha_2} \cdots e_{\alpha_p}
\right\rangle,
\end{equation}
which is also equal to
\begin{equation}
\left\langle e_j^{\perp} C_{n,k}, 
e_{\alpha_2} \cdots e_{\alpha_p}
\right\rangle,
\end{equation}
completing the proof.
\end{proof}

We will  need the theory of hypergeometric series for our proof of
Equation~\eqref{delta-zero-equation}.  In particular, we 
have the following transformation of the  ${ _3 \phi_2}$ basic 
hypergeometric series
(see \cite{AndrewsAskeyRoy} for background on basic 
hypergeometric series).

\begin{lemma}
\label{hypergeometric-lemma}
Let $j \in \NN$ and $\alpha, x, y, z \in \RR$.  We have
\begin{align}
\label{trans}
{ _3\phi _2} \left ( \begin{matrix} q^{-j},& q^{\alpha},& q^{\alpha +z} \\ & q^{\alpha-y-j+1}, &q^{\alpha-x-j+1} \end{matrix} ;q,q\right ) \\
\notag =
\frac { (q^{-y-j+1})_j (q^{-x-j+1})_j }{ (q^{\alpha -y-j+1})_j(q^{\alpha -x-j+1})_j} 
\frac {(q^{-y-z-j+1})_j}{(q^{y})_j} q^{(\alpha + x+y+z+j-1)j}
{ _3\phi _2} \left ( \begin{matrix} q^{-j},& q^{x+y+z+j-1},& q^{x-\alpha} \\ & q^{x}, &q^{x+z} \end{matrix} ;q,q^{1+\alpha -y}\right ).
\end{align}
\end{lemma}

\begin{proof}
We utilize the following identities from \cite[p. 525]{AndrewsAskeyRoy}:
\begin{align}
\label{AAR1}
_3\phi _2 \left ( \begin{matrix} q^{-n},& w,& b \\ & d, &e \end{matrix} ;q,q\right )= 
\frac {(e/w)_n w^n}{(e)_n} 
{ _3\phi _2} \left ( \begin{matrix} q^{-n},& w,& d/b \\ & d, &q^{1-n}w/e \end{matrix} ;q,bq/e\right ),\\
\label{AAR2}
{ _3\phi _2} \left ( \begin{matrix} q^{-n},& w,& b \\ & e, &f \end{matrix} ;q,q\right )= 
\frac {(e/w)_n (f/w)_n w^n}{(e)_n(f)_n} 
{ _3\phi _2} \left ( \begin{matrix} q^{-n},& w,& wbq^{1-n}/ef \\ & q^{1-n}w/e, &q^{1-n}w/f \end{matrix} ;q,q\right ),
\end{align}
where $n \in \mathbb N$, and $w,b,c,d,e,f$ are continuous parameters.  

Begin by setting 
$j=n, w=q^{\alpha}, b=q^{\alpha +z}, e=q^{\alpha -y-j+1}$, and $f=q^{\alpha -x-j+1}$ in (\ref{AAR2}) to get
\begin{align}
\label{tran1}
{ _3\phi _2} \left ( \begin{matrix} q^{-j},& q^{\alpha},& q^{\alpha +z} \\ & q^{ \alpha-y-j+1}, &q^{\alpha -x-j+1} \end{matrix} ;q,q\right )= 
\frac { (q^{-y-j+1} )_j ( q^{-x-j+1} )_j q^{j\alpha} }{ (q^{\alpha -y-j+1})_j(q^{\alpha -x-j+1})_j } 
{ _3\phi _2} \left ( \begin{matrix} q^{-j},& q^{\alpha},& q^{x+y+z+j-1} \\ & q^{y}, &q^{x}\end{matrix} ;q,q\right ).
\end{align}
Now apply (\ref{AAR1}) with $n=j, w=q^{x+y+z+j-1}, b = q^{\alpha}, d=q^{x}, e=q^{y}$ to the 
${  _3 \phi _2}$ appearing in the RHS of (\ref{tran1}) to get
\begin{align}
\label{tran2}
{ _3\phi _2} \left ( \begin{matrix} q^{-j},& q^{\alpha},& q^{\alpha +z} \\ & q^{ \alpha-y-j+1}, &q^{\alpha -x-j+1} \end{matrix} ;q,q\right )
\\
\notag = 
\frac { (q^{-y-j+1})_j (q^{-x-j+1})_j }{ (q^{\alpha -y-j+1})_j(q^{\alpha -x-j+1})_j} 
\frac {(q^{-x-z-j+1})_j}{(q^{y})_j} q^{(\alpha + x+y+z+j-1)j}
{ _3\phi _2} \left ( \begin{matrix} q^{-j},& q^{x+y+z+j-1},& q^{x-\alpha} \\ & q^{x}, &q^{x+z} \end{matrix} ;q,q^{1+\alpha -y}\right ).
\end{align}
\end{proof}

We express the hypergeometric transformation of Lemma~\ref{hypergeometric-lemma}
in a more convenient form involving $q$-binomials.

\begin{lemma}
\label{q-binomial-lemma}
Let $j \leq k \leq n$ be positive integers.  Let $p$ be an integer in the range
$k-j \leq p \leq n-j$.  There holds the identity
\begin{align}
\label{Brendon}
q^{{k \choose 2} +{j \choose 2}} \sum_{r=p}^{p+j}(-1)^{n-r} 
\left [ \begin{matrix} r-1 \\ k-1 \end{matrix} \right ] _q 
q^{{r+1 \choose 2} -nr}  
\left [ \begin{matrix} r \\ j \end{matrix} \right ] _q 
q^{(r-p)(n-j-p)}
\left [ \begin{matrix} j \\ r-p \end{matrix} \right ] _q  = \\
\notag
\sum_{r=k-p}^j q^{{r \choose 2}}
\left [ \begin{matrix} k \\ r \end{matrix} \right ] _q  
\left [ \begin{matrix} k+j-r-1 \\ j-r \end{matrix} \right ] _q  
q^{{k-a \choose 2}} (-1)^{n-j-p}
\left [ \begin{matrix} p-1 \\ k-r-1 \end{matrix} \right ] _q  
q^{{p+1 \choose 2}-(n-j)p}.
\end{align} 
\end{lemma}

\begin{proof}
The first step is to express everything in terms of hypergeometric series.  
We make use 
of the following facts, which we refer to as the `simple identities'.  Here 
$u, j, a \in \ZZ_{\geq 0}$ and $p, x$ are continuous parameters.
\begin{align}
{u + p \brack j}_q &= {p \brack j}_q \frac{(q^{p+1})_u}{(q^{p-j+1})_u} \\
{p \brack u + a}_q = {p \brack a}_q \frac{(q^{p-a-u+1})_u}{(q^{a+1})_u} &=
{p \brack a}_q \frac{(q^{a-p})_u}{(q^{a+1})_u} (-q^{p-a})^u q^{-{u \choose 2}} \\
{p \brack j-u}_q &= {p \brack j}_q \frac{(q^{-j})_u}{(q^{p-j+1})_u} 
(-q^j)^u q^{-{u \choose 2}} \\
{u + a \choose 2} &= {u \choose 2} + {a \choose 2} + ua.
\end{align}

Using the simple identities and setting $u = r-p$, the LHS of Equation~\eqref{Brendon}
can be expressed as 
\begin{equation}
\label{BrendonLeft}
{p \brack j}_q {p-1 \brack k-1}_q (-1)^{n-p} q^{-np + {p+1 \choose 2} + {k \choose 2} + {j \choose 2}}
\sum_{u = 0}^j \frac{(q^{-j}, q^{p+1}, q^p)_u}{(q,q^{p-j+1},q^{p-k+1})_u} q^u.
\end{equation}
Similarly, using the simple identities and setting $u = j-a$, we see that the RHS
of Equation~\eqref{Brendon} can be expressed as
\begin{equation}
\label{BrendonRight}
q^{-p(n-j) + {p+1 \choose 2} + {k - j \choose 2} + {j \choose 2}} (-1)^{n+j-p}
{k \brack j}_q {p-1 \brack k-j-1}_q
\sum_{u = 0}^{p-k+j} \frac{(q^{-j}, q^k, q^{-p+k-j})_u}{(q,q^{k-j},q^{k-j+1})_u}
q^{u(p+1)}.
\end{equation}

The next step is to express \eqref{BrendonLeft} and \eqref{BrendonRight}
in terms of the hypergeometric series ${}_3 \phi_2$.
The expression \eqref{BrendonLeft} is given by
\begin{equation}
\label{BrendonLeft2}
{p \brack j}_q {p-1 \brack k-1}_q (-1)^{n-p}
q^{-np + {p+1 \choose 2} + {k \choose 2} + {j \choose 2}}
{ _3\phi _2} \left ( \begin{matrix} q^{-j},& q^{p+1},& q^{p} \\ & q^{p-j+1}, &q^{p-k+1} \end{matrix} ;q,q\right )
\end{equation}
whereas \eqref{BrendonRight} is equal to
\begin{equation}
\label{BrendonRight2}
q^{-p(n-j) + {p+1 \choose 2} + {k - j \choose 2} + {j \choose 2}}
(-1)^{n+j-p} {k \brack j}_q {p-1 \brack k-j-1}_q 
{ _3\phi _2} \left ( \begin{matrix} q^{-j},& q^{k},& q^{-p+k-j} \\ & q^{k-j}, &q^{k-j+1} \end{matrix} ;q,q^{p+1}\right ).
\end{equation}
The fact that \eqref{BrendonLeft2} = \eqref{BrendonRight2} is a consequence 
of Lemma~\ref{hypergeometric-lemma}.
\end{proof}

\section{Proofs of the Main Results}
\label{Proofs}

Our starting point is the following expansion
(see \cite[Eqn. 2.72]{Hagbook}) of $e_n$ in the modified Macdonald basis:
\begin{equation}
\label{macdonald-expansion}
e_n = \sum_{\lambda \vdash n} 
\frac{M B_{\lambda} \Pi_{\lambda} \widetilde{H}_{\lambda}}{w_{\lambda}},
\end{equation}
where 
\begin{itemize}
\item $M = (1-q)(1-t)$,
\item $B_{\lambda} = \sum_{(i,j) \in \lambda} q^{i-1} t^{j-1}$, where the sum is over
all cells $(i,j)$ in the Ferrers diagram of $\lambda$,
\item  $\Pi_{\lambda} = \prod_{(1,1) \neq (i,j) \in \lambda}
(1 - q^{i-1} t^{j-1})$, where the product is over all cells $(i,j)$ in the Ferrers diagram of 
$\lambda$ other than the northwest corner $(1,1)$, and
\item  
$w_{\lambda} = \prod_{c \in \lambda} (q^{a(c)} - t^{l(c) + 1})(t^{l(c)} - q^{a(c)+1})$,
where the product is over all cells $c$ in the diagram of $\lambda$ and $a(c), l(c)$
are the arm and leg lengths of the cell $c$ in $\lambda$.
\end{itemize}

If we apply the operator $\Delta'_{e_{k-1}}$ to both sides of 
Equation~\eqref{macdonald-expansion},
we get
\begin{equation}
\label{delta-expansion}
\Delta'_{e_{k-1}} e_n = \sum_{\lambda \vdash n} e_{k-1}[B_{\lambda} - 1]
\frac{M B_{\lambda} \Pi_{\lambda} \widetilde{H}_{\lambda}}{w_{\lambda}}.
\end{equation}
Here we used the plethystic shorthand
$e_{k-1}[B_{\lambda} - 1] = e_{k-1}(\dots, q^{i-1} t^{j-1}, \dots )$ where $(i,j)$ range over all cells 
$\neq (1,1)$ in the Ferrers diagram of $\lambda$.

Recall that $\widetilde{H}_{\lambda} |_{t = 0} = \rev_q Q'_{\lambda}$ for any partition
$\lambda$.
If we evaluate both sides of Equation~\eqref{delta-expansion} at $t = 0$, we get 
\begin{equation}
\label{delta-expansion-zero}
\Delta'_{e_{k-1}} e_n|_{t = 0} = 
\sum_{\lambda \vdash n} (-1)^{n - \ell(\lambda)}
q^{{k \choose 2} - 2 b(\lambda) - n + 
\sum_i {m_i(\lambda) + 1 \choose 2}} 
{\ell(\lambda) - 1 \brack k-1}_q 
{\ell(\lambda) \brack m(\lambda)}_q \cdot
\rev_q Q'_{\lambda}.
\end{equation}
Here we used the evaluation
\begin{equation}
e_{k-1}[B_{\lambda} - 1] = q^{{k \choose 2}} {\ell(\lambda) - 1 \brack k-1}_q.
\end{equation}
Equation~\eqref{delta-expansion-zero} can be expressed in terms of the $D$-functions
$D_{n,r}$.

\begin{lemma}
\label{new-delta-expansion}
We have the identity
\begin{equation}
\Delta'_{e_{k-1}} e_n |_{t = 0} = q^{k \choose 2} \sum_{r = k}^n (-1)^{n-r}  
q^{{r+1 \choose 2} - nr}
{r - 1 \brack k-1}_q 
D_{n,r}.
\end{equation}
\end{lemma}

\begin{proof}
Starting with Equation~\eqref{delta-expansion-zero} and grouping partitions $\lambda \vdash n$ according to their number 
of parts we have
\begin{align}
\Delta'_{e_{k-1}} e_n |_{t = 0} &= \sum_{\lambda \vdash n}
(-1)^{n - \ell(\mu)} q^{{k \choose 2} - 2 b(\mu) - n + \sum_i {m_i(\lambda) + 1 \choose 2}}
{\ell(\lambda) - 1 \brack k-1}_q {\ell(\lambda) \brack m(\mu)}_q \cdot \rev_q Q'_{\lambda} \\
&= q^{{k \choose 2} - n} \sum_{r = k}^n (-1)^{n-r} {r-1 \brack k-1}_q
\sum_{\substack{\lambda \vdash n \\ \ell(\mu) = r}}
q^{-2 b(\mu) + \sum_i {m_i(\lambda) + 1 \choose 2}} {r \brack m(\lambda)}_q \cdot \rev_q Q'_{\lambda}.
\end{align}

We focus on the internal summand.  We have
\begin{align}
\sum_{\substack{\mu \vdash n \\ \ell(\lambda) = r}}
&q^{-2 b(\lambda) + \sum_i {m_i(\lambda) + 1 \choose 2}} {r \brack m(\lambda)}_q \cdot \rev_q Q'_{\lambda} \\ &=
\sum_{\substack{\lambda \vdash n \\ \ell(\lambda) = r}}
q^{-2 b(\mu) + \sum_i {m_i(\lambda) + 1 \choose 2} + \sum_{i < j} m_i(\lambda) m_j(\lambda) -
\sum_{i < j} m_i(\lambda) m_j(\lambda)} {r \brack m(\lambda)}_q \cdot \rev_q Q'_{\lambda} \\
&=  q^{r + 1 \choose 2}  \sum_{\substack{\lambda \vdash n \\ \ell(\lambda) = r}}
q^{-2 b(\lambda) - \sum_{i < j} m_i(\lambda) m_j(\lambda)} {r \brack m(\lambda)}_q \cdot \rev_q Q'_{\lambda} \\
&= q^{r + 1 \choose 2}   \sum_{\substack{\lambda \vdash n \\ \ell(\lambda) = r}}
q^{-b(\lambda)} \cdot \left[  q^{\sum_{i < j} -m_i(\lambda) m_j(\lambda)} {r \brack m(\lambda)}_q \right] \cdot
\left[q^{-b(\lambda)} \rev_q Q'_{\lambda}(\xx;q) \right] \\
&= q^{{r + 1 \choose 2} - {r \choose 2}}  \sum_{\substack{\lambda \vdash n \\ \ell(\lambda) = r}}
q^{-\overline{b}(\lambda)} \cdot {r \brack m(\lambda)}_{q^{-1}} \cdot Q'_{\lambda}(\xx; q^{-1}) \\
&= q^{{r + 1 \choose 2} - {r \choose 2}}
q^{-{r \choose 2} - (n-r)(r-1)} \cdot 
[  q^{{r \choose 2} + (n-r)(r-1)} \cdot \omega C_{n,r}(\xx;q^{-1}) ] \\
&= q^{{r + 1 \choose 2} - r(r-1) - (n-r)(r-1)}
 \cdot 
D_{n,r}(\xx;q) \\
&= q^{{r + 1 \choose 2} - nr + n}
 \cdot 
D_{n,r}(\xx;q) 
\end{align}
The second equality used $\sum_i {m_i(\lambda) + 1 \choose 2} + \sum_{i < j} m_i(\lambda) m_j(\lambda) = {\ell(\lambda) +1 \choose 2}$.
The fourth equality comes from the fact that the degree of the palindromic 
polynomial ${\ell(\lambda) \brack m(\lambda)}_q$ is 
$\sum_{i < j} m_i(\lambda) m_j(\lambda)$ and that the $q$-degree of $Q'_{\lambda}(\xx;q)$ is $b(\lambda)$.

Going back to the assertion of the lemma, we have
\begin{align}
\Delta'_{e_{k-1}} e_n |_{t = 0} &= 
q^{{k \choose 2} - n} \sum_{r = k}^n (-1)^{n-r} {r-1 \brack k-1}_q
\sum_{\substack{\lambda \vdash n \\ \ell(\lambda) = r}}
q^{-2 b(\mu) + \sum_i {m_i(\lambda) + 1 \choose 2}} {r \brack m(\lambda)}_q \cdot \rev_q Q'_{\lambda} \\
&=  q^{{k \choose 2} - n} \sum_{r = k}^n (-1)^{n-r} {r-1 \brack k-1}_q \cdot
q^{{r + 1 \choose 2} - nr + n}
 \cdot 
D_{n,r}(\xx;q).
\end{align}
Canceling a factor of $q^n$ completes the proof.
\end{proof}

We are in a position to give our proof of Equation~\eqref{delta-zero-equation},
and thus give a new proof of the Delta Conjecture at $t = 0$.

\begin{theorem}
\label{delta-zero}
(Garsia-H.-Remmel-Yoo \cite{GHRY})
Let $k \leq n$ be positive integers.  We have
\begin{equation}
\Delta'_{e_{k-1}} e_n|_{t = 0} = \Delta'_{e_{k-1}} e_n|_{q = 0, t = q} =  C_{n,k}.
\end{equation}
\end{theorem}

\begin{proof}
Let $j \geq 1$.
Given Lemma~\ref{new-delta-expansion}
and Lemma~\ref{d-e-perp-lemma}, the symmetric function 
$e_j^{\perp} \Delta'_{e_{k-1}} e_n|_{t = 0}$
has the following $D$-function expansion, where we adopt the convention that
$D_{n,k} = 0$ if $k > n$ or if $k < 0$:
\begin{small}
\begin{equation}
\label{what-we-know}
e_j^{\perp}
\Delta'_{e_{k-1}} e_n|_{t = 0} =
q^{{k \choose 2} + {j \choose 2}} \sum_{r = k}^n (-1)^{n-r}
{r-1 \brack k-1}_q q^{{r+1 \choose 2} - nr} {r \brack j}_q
\sum_{m = r-j}^r q^{(r-m)(n-j-m)} {j \brack r-m}_q 
D_{n-j,m}.
\end{equation}
\end{small}
If we want $e_j^{\perp} \Delta'_{e_{k-1}} e_n|_{t = 0}$ to satisfy the recursion
of Lemma~\ref{c-e-perp-recursion}, we must have
\begin{equation}
\label{what-we-want}
e_j^{\perp} \Delta'_{e_{k-1}} e_n|_{t = 0} = \sum_{r = 0}^j
q^{{r \choose 2}} {k \brack r}_q {k + j - r - 1 \brack j - r}_q 
\Delta'_{e_{k-r-1}} e_{n-j}|_{t = 0}.
\end{equation}
By
Lemma~\ref{new-delta-expansion}, we know
\begin{equation}
\label{also-what-we-know}
\begin{split}
\sum_{r = 0}^j &q^{r \choose 2} {k \brack a}_q {k + j - r - 1 \brack j - r}_q \Delta'_{e_{k-r-1}} e_{n-j} |_{t = 0} = \\
&\sum_{r = 0}^j q^{r \choose 2} {k \brack r}_q {k + j - r - 1 \brack j - r}_q q^{k - r \choose 2}
\sum_{b = k - r}^{n-j} (-1)^{n-j-b} {b - 1 \brack k-r-1}_q 
q^{{b + 1 \choose 2} - (n-j)b} 
D_{n-j,b}.
\end{split}
\end{equation}

We want to show that the RHS of Equation~\eqref{what-we-want} is equal to the RHS of
Equation~\eqref{also-what-we-know}.
To this end, let $p$ be an integer in the range $k-j \leq p \leq n-j$.  The coefficient of $D_{n-j,m}$ in 
Equation~\eqref{what-we-want} is
\begin{equation}
\label{d-expression-one}
q^{ {k \choose 2} + {j \choose 2}}  \sum_{r = p}^{p+j} (-1)^{n-r} {r-1 \brack k-1}_q
q^{{r + 1 \choose 2} - nr} {r \brack j}_q 
q^{(r-p)(n-j-p)} {j \brack r-p}_q
\end{equation}
whereas the coefficient of $D_{n-j,p}$ in Equation~\ref{also-what-we-know} is 
\begin{equation}
\label{d-expression-two}
\sum_{r = k-p}^j q^{r \choose 2} {k \brack r}_q {k + j - r - 1 \brack j - r}_q q^{k - r \choose 2}
(-1)^{n-j-p}
 {p - 1 \brack k-r-1}_q 
q^{{p + 1 \choose 2} - (n-j)p}.
\end{equation}
Theorem~\ref{delta-zero} will be proven if we can only establish the equality
of the expressions \eqref{d-expression-one}
and \eqref{d-expression-two}.
This is Lemma~\ref{q-binomial-lemma}.
\end{proof}

We use Theorem~\ref{delta-zero} to derive the more general
Theorem~\ref{delta-hl-expansion-theorem}.  In this proof we will
use the notation of plethysm; see \cite{Hagbook}.

\begin{proof}  (of Theorem~\ref{delta-hl-expansion-theorem})
Let $k \leq n$ be positive integers.
The polynomials $Q'_{\mu}$ and $\rev_q(Q'_{\mu})$ have Schur expansions
\begin{align}
Q'_{\mu} &= \sum_{\lambda \vdash n} K_{\lambda,\mu}(q) s_{\lambda}, \\
\rev_q(Q'_{\mu}) &= \sum_{\lambda \vdash n} q^{b(\mu)} K_{\lambda,\mu}(1/q) s_{\lambda}.
\end{align}
By Equation~\eqref{c-hl-expansion}, Equation~\eqref{delta-expansion-zero},
and the truth of the Delta Conjecture at $q = 0$ (i.e., Theorem~\ref{delta-zero}) we have the 
following identity.
\begin{align}
\label{DC0}
q^{k(k-1)}\sum_{\mu\vdash n} (-1)^{n-\ell(\mu)}q^{-n-b(\mu)+\sum_{i=1}^n \binom{m_i +1}{2}}\begin{bmatrix} \ell(\mu) -1\\k-1\end{bmatrix}_q
\left [ \begin{matrix} \ell(\mu) \\ m(\mu) \end{matrix} \right ]_q 
K_{\lambda,\mu}(1/q)\\
\notag
=\sum_{\mu\vdash n \atop \ell(\mu)=k}q^{b(\mu)}
\left [ \begin{matrix} \ell(\mu) \\ m(\mu) \end{matrix} \right ]_q 
K_{\lambda ',\mu} (q)
\end{align}
Equation~\eqref{DC0} is also recorded in \cite[Prop. 3.2]{GHRY}.

Using reasoning identical to that of our derivation of Equation~\eqref{delta-expansion-zero},
we see that $ \Delta'_{s_{\nu}} e_n|_{t = 0}$ has the following expansion
in the $q$-reversed $Q'$-basis.
\begin{equation}
\label{schur-delta-expansion-first}
\Delta'_{s_{\nu}} e_n|_{t = 0} = \sum_{\mu \vdash n}
(-1)^{n - \ell(\mu)} s_{\nu}(q, q^2, \dots, q^{\ell(\mu) - 1})
q^{-n - 2 b(\mu) + \sum_i {m_i(\mu) + 1 \choose 2}}
{\ell(\mu) \brack m(\mu)}_q \cdot \rev_q(Q'_{\mu}).
\end{equation}

Multiplying both sides of (\ref{DC0}) by $s_{\lambda}$, 
summing over $\lambda$ and applying $\omega$ we get the following equivalent form of
(\ref{DC0}):
\begin{align}
\label{DC1}
q^{k(k-1)}\sum_{\mu\vdash n} (-1)^{n-\ell(\mu)}q^{-n-b(\mu)+\sum_{i=1}^n \binom{m_i +1}{2}}\begin{bmatrix} \ell(\mu) -1\\k-1\end{bmatrix}_q
\left [ \begin{matrix} \ell(\mu) \\ m(\mu) \end{matrix} \right ]_q 
\sum_{\lambda}  K_{\lambda ^{\prime},\mu}(1/q) s_{\lambda} \\
\notag
=\sum_{\mu\vdash n \atop \ell(\mu)=k}q^{b(\mu)}
\left [ \begin{matrix} \ell(\mu) \\ m(\mu) \end{matrix} \right ]_q 
Q^{\prime}_{\mu}
\end{align}
for all $\lambda \vdash n$ and $1\le k\le n$.

Note that the sum on the RHS of (\ref{DC1}) also occurs on the RHS in 
Theorem~\ref{delta-hl-expansion-theorem}.  
By Equation~\eqref{schur-delta-expansion-first} the following 
equation is equivalent to Theorem~\ref{delta-hl-expansion-theorem}.
 \begin{align}
 \label{DC2}
 \sum_{\mu \vdash n} (-1)^{n - \ell (\mu)} q^{-b(\mu) -n + \sum_{i} {m_i(\mu)+1 \choose 2} }
 \left [ \begin{matrix} \ell(\mu) \\ m(\mu) \end{matrix} \right ]_q 
 \sum_{\lambda} s_{\lambda} K_{\lambda ^{\prime},\mu}(1/q) s_{\nu}(q,q^2,\ldots ,q^{\ell(\mu)-1})  \\
 \notag
 = \sum_{\mu \vdash n}  (-1)^{n-\ell(\mu)}                     
 q^{-{\ell(\mu) \choose 2}  -{ \ell(\mu) \choose 2} + \ell(\mu)(\ell(\mu)-1) -n -b(\mu) + 
 \sum_{i} {m_i(\mu)+1 \choose 2}}
  \left [ \begin{matrix} \ell(\mu) \\ m(\mu) \end{matrix} \right ]_q 
 \sum_{\lambda} s_{\lambda} K_{\lambda ^{\prime},\mu} (1/q) \\
 \notag
 \times
 \sum_{k,\rho \atop \ell(\rho)=k-1, \, |\rho|=|\nu|} q^{|\nu|+b(\rho)}
 \left [ \begin{matrix} \ell(\mu)-1 \\ k-1 \end{matrix} \right ]_q 
 \left [ \begin{matrix} k-1 \\ m(\rho) \end{matrix} \right ]_q K_{\nu, \rho}(q).
 \end{align}
If we can show the coefficients of $s_{\lambda} K_{\lambda ^{\prime},\mu}(1/q)$ in the inner sums on both sides of (\ref{DC2}) are equal for any $\mu \vdash n$ then (\ref{DC2}), and hence 
Theorem~\ref{delta-hl-expansion-theorem}, will follow.   Replacing $\ell(\mu)$ by $j+1$ this statement can be expressed as
 \begin{align}
 \label{simple2}
 s_{\nu}(1,q,q^2,\ldots ,q^{j-1})  =
 \sum_{k,\rho \atop \ell(\rho)=k-1, |\rho|=|\nu|} 
 q^{b(\rho)}
  \left [ \begin{matrix} j  \\ k-1 \end{matrix} \right ]_q 
 \left [ \begin{matrix} k-1 \\ m(\rho) \end{matrix} \right ]_q
 K_{\nu,\rho}(q),
 \end{align}
 for any nonnegative integer $j$.
 
 To prove (\ref{simple2}), multiply both sides of (\ref{simple2}) 
 by $s_{\nu} = s_{\nu}[\xx]$ 
 and sum over $\nu$.   Using the Cauchy identity, (\ref{simple2}) is thus
 equivalent to
 \begin{align}
 \label{simple3}
h_n \left[ (1-q^j)\frac {\xx}{1-q} \right] =
\sum_{k,\rho \atop \ell(\rho)=k-1, \, |\rho|=|\nu|} 
 q^{b(\rho)}
  \left [ \begin{matrix} j  \\ k-1 \end{matrix} \right ]_q 
 \left [ \begin{matrix} k-1 \\ m(\rho) \end{matrix} \right ]_q
 \sum_{\nu} s_{\nu}[\xx] K_{\nu,\rho}(q),
 \end{align}
 for any nonnegative integer $j$.  Using  
 \cite[Eqn. 14]{GHRY} this can be expressed as
 \begin{align}
 \label{simple4}
h_n\left[(1-q^j)\frac {\xx}{1-q} \right] =
\sum_{\rho , \, |\rho|=|\nu|} 
 q^{n(\rho)}
  \left [ \begin{matrix} j  \\ \ell(\rho) \end{matrix} \right ]_q 
  [\ell(\rho)]!(1-q)^{\ell(\rho)}
P_{\rho}\left[ \frac{\xx}{1-q};q \right],
 \end{align}
for any nonnegative integer $j$.   
Here $P_{\rho}$ is the Hall-Littlewood $P$-function.
Making the transformations $\xx \mapsto \xx/(1-q)$ and
$\yy \mapsto 1 - q^j$ in
\cite[Eqn. 17]{GHRY}, another expression for the LHS of (\ref{simple4})
is
\begin{align}
\label{Cauchy2}
\sum_{\rho , \, |\rho|=|\nu|} 
P_{\rho}\left[ \frac{\xx}{1-q};q \right] \, Q^{\prime}_{\rho}[1-q^j;q].
\end{align}
By \cite[Lem. 3.3]{GHRY} we have
\begin{align}
\label{Lemma33}
Q^{\prime}_{\rho}[1-q^j;q] = q^{b(\rho)}(1-q^j)(1-q^{j-1})\cdots (1-q^{j-\ell(\rho)+1}).
\end{align}
Using (\ref{Lemma33}) in (\ref{Cauchy2}) and simplifying we see the RHS of (\ref{simple4}) is the same as (\ref{Cauchy2}), which completes the proof.
\end{proof}

We want to prove the algebraic interpretation of 
$\Delta'_{s_{\nu}} e_n$ at $t = 0$ given in
Theorem~\ref{algebraic-interpretation}.
This interpretation is based on the $q$-reversal of the following symmetric function identity.
We consider symmetric functions in two infinite variable sets: $\xx$ and $\yy$.
We let $\omega_{\xx}$ and $\omega_{\yy}$ be the $\omega$ involution acting
on the $\xx$ and $\yy$ variables (respectively).

\begin{proposition}
\label{omega-to-c}
Let $n, m \geq 0$.  We have
\begin{equation}
\label{omega-to-c-equation}
\sum_{\nu \vdash m} s_{\nu}(\yy) \cdot \omega_{\xx} \Delta'_{s_{\nu}} e_n(\xx) |_{t = 0} = 
 \sum_{k \geq 0} q^{m-k+1} \omega_{\yy} C_{m, k-1}(\yy; q) \cdot \omega_{\xx} C_{n, k}(\xx; q).
\end{equation}
\end{proposition}

\begin{proof}
By Theorem~\ref{delta-hl-expansion-theorem},
\begin{align}
\sum_{\nu \vdash m} &s_{\nu}(\yy) \cdot \omega_{\xx} \Delta'_{s_{\nu}} e_n(\xx) |_{t = 0} \\ &= 
\sum_{\nu \vdash m} s_{\nu}(\yy) \cdot \omega_{\xx} \sum_{k = \ell(\nu) + 1}^{|\nu| + 1} P_{\nu, k-1}(q) 
\sum_{\substack{\mu \vdash n \\ \ell(\mu) = k}} q^{\overline{b}(\mu)} {k \brack m(\mu)}_q Q'_{\mu}(\xx; q) \\
&= 
\sum_{\nu \vdash m} s_{\nu}(\yy) \cdot \omega_{\xx} \sum_{k = \ell(\nu) + 1}^{|\nu| + 1} P_{\nu, k-1}(q) 
\cdot \omega_{\xx} C_{n,k}(\xx; q) \\
&=
\sum_{\nu \vdash m} \sum_{k = \ell(\nu) + 1}^{|\nu| + 1} 
q^{|\nu| - {k \choose 2}} \sum_{\substack{\rho \vdash m \\ \ell(\rho) = k-1}}
q^{b(\rho)} {k -1 \brack m(\rho)}_q K_{\nu, \rho}(q) s_{\nu}(\yy) 
\cdot \omega_{\xx} C_{n,k}(\xx; q) \\
&=
\sum_{\nu \vdash m} \sum_{k = \ell(\nu) + 1}^{|\nu| + 1} 
q^{m - k + 1} \sum_{\substack{\rho \vdash m \\ \ell(\rho) = k-1}}
q^{\overline{b}(\rho)} {k -1 \brack m(\rho)}_q K_{\nu, \rho}(q) s_{\nu}(\yy) 
\cdot \omega_{\xx} C_{n,k}(\xx; q) \\
&=
\sum_{\substack{\nu \vdash m \\ k \geq 1}}
q^{m - k + 1} \sum_{\substack{\rho \vdash m \\ \ell(\rho) = k-1}}
q^{\overline{b}(\rho)} {k -1 \brack m(\rho)}_q K_{\nu, \rho}(q) s_{\nu}(\yy) 
\cdot \omega_{\xx} C_{n,k}(\xx; q) \\
&= \sum_{k \geq 1} q^{m-k+1} \sum_{\substack{\rho \vdash m \\ \ell(\rho) = k-1}}
q^{\overline{b}(\rho)} {k -1 \brack m(\rho)}_q 
\left[ \sum_{\nu \vdash m} K_{\nu, \rho}(q) s_{\nu}(\yy)  \right] 
\cdot \omega_{\xx} C_{n,k}(\xx; q) \\
&= \sum_{k \geq 1} q^{m-k+1} \sum_{\substack{\rho \vdash m \\ \ell(\rho) = k-1}}
q^{\overline{b}(\rho)} {k -1 \brack m(\rho)}_q 
Q'_{\rho}(\yy; q)
\cdot \omega_{\xx} C_{n,k}(\xx; q) \\
&= \sum_{k \geq 1} q^{m-k+1} \omega_{\yy} C_{m, k-1}(\yy; q)
\cdot \omega_{\xx} C_{n,k}(\xx; q),
\end{align}
which is what we wanted to prove.
\end{proof}

We want to $q$-reverse the identity of Proposition~\ref{omega-to-c}.

\begin{proposition}
\label{omega-to-d}
Let $n, m \geq 0$.  We have
\begin{equation}
\label{omega-to-d-equation}
\sum_{\nu \vdash m} 
q^{b(\nu)} s_{\nu}(\yy) \cdot (\rev_q \circ \omega_{\xx}) \Delta'_{s_{\nu}} e_n(\xx) |_{t = 0} = 
 \sum_{k \geq 0} q^{mn - km - kn + n + k(k-1)}
 D_{m,k-1}(\yy;q) \cdot D_{n,k}(\xx;q).
\end{equation}
\end{proposition}

\begin{proof}
We want to $q$-reverse both sides of Equation~\eqref{omega-to-c-equation}.  We begin with the
LHS.

{\bf Claim:}  {\em For any partition $\nu \vdash m$, the $q$-degree of 
$\Delta'_{s_{\nu}} e_n|_{t = 0}$ is $(n-1)m - b(\nu)$. }

To see why the Claim is true, let $\nu \vdash m$ and consider
Equation~\eqref{schur-delta-expansion-first}, recapitulated here:
\begin{equation*}
\Delta'_{s_{\nu}} e_n|_{t = 0} = \sum_{\mu \vdash n}
(-1)^{n - \ell(\mu)} s_{\nu}(q, q^2, \dots, q^{\ell(\mu) - 1})
q^{-n - 2 b(\mu) + \sum_i {m_i(\mu) + 1 \choose 2}}
{\ell(\mu) \brack m(\mu)}_q \cdot \rev_q(Q'_{\mu}).
\end{equation*}
We know that $Q'_{\mu}$ (and also $\rev_q(Q'_{\mu})$) has $q$-degree
$b(\mu)$.  If $\mu \vdash n$ is such that $\ell(\mu) > \ell(\nu)$ so that 
the $\mu$-summand on the RHS of
Equation~\eqref{schur-delta-expansion-first} does not vanish,
the $q$-degree of this $\mu$-summand is therefore
\begin{equation}
\label{q-degree-equation}
\sum_{i = j}^{\ell(\nu)} \nu_j(\ell(\mu) - j) - n - 2 b(\mu) + 
\sum_i {m_i(\mu) + 1 \choose 2} + \sum_{i < j} m_i(\mu) m_j(\mu) + b(\mu),
\end{equation}
or equivalently
\begin{equation}
\label{q-degree-equation-two}
m(\ell(\mu) - 1) - b(\nu) - n -b(\mu) +
\sum_i {m_i(\mu) + 1 \choose 2} + \sum_{i < j} m_i(\mu) m_j(\mu).
\end{equation}
It is not hard to see that Expression~\eqref{q-degree-equation-two} is maximized
uniquely when $\mu = (1^n)$, in which case it equals
\begin{equation}
\label{q-degree-equation-three}
m(n - 1) - b(\nu) - n - {n \choose 2} + {n+1 \choose 2} = 
m(n-1) - b(\nu),
\end{equation}
which completes the proof of the Claim.

Our Claim implies that the overall $q$-degree of 
the LHS of Equation~\ref{omega-to-c-equation} (and hence also the RHS)
is $m(n-1)$; this corresponds to the summand $\nu = (m)$ so that $b(\nu) = 0$.  
The $q$-reversal of the LHS of Equation~\eqref{omega-to-c-equation} 
is therefore
\begin{equation}
q^{m(n-1)} \sum_{\nu \vdash m} q^{b(\nu) - (n-1)m} s_{\nu}(\yy) \cdot
(\rev_q \circ \omega_{\xx}) \Delta'_{s_{\nu}} e_n(\xx) |_{t = 0},
\end{equation}
which coincides with the LHS of Equation~\eqref{omega-to-d-equation}.

Now we $q$-reverse the RHS of Equation~\eqref{omega-to-c-equation}.
Since the $q$-degree of $C_{m,k-1}(\yy;q)$ is $(k-2)m - {k-1 \choose 2}$ and the 
$q$-degree of $C_{n,k}(\xx;q)$ is $(k-1)n - {k \choose 2}$, the $q$-reversal of the 
RHS of Equation~\eqref{omega-to-c-equation} is 
\begin{equation}
q^{m(n-1)} \sum_{k \geq 0} q^{-m+k-1} \cdot 
\left[ q^{-(k-2)m + {k-1 \choose 2}} D_{m,k-1}(\yy;q) \right] \cdot
\left[ q^{-(k-1)n + {k \choose 2}} D_{n,k}(\xx;q) \right],
\end{equation}
which is equivalent to the RHS of 
Equation~\eqref{omega-to-d-equation}.
\end{proof}

We are in a position to prove Theorem~\ref{algebraic-interpretation}.

\begin{proof}
(of Theorem~\ref{algebraic-interpretation})
If $\lambda \vdash m$ and $\mu \vdash n$ are any partitions, the
Frobenius image of the irreducible $\symm_m \times \symm_n$-module
$S^{\lambda} \otimes S^{\mu}$ is
$\Frob(S^{\lambda} \otimes S^{\mu}) = s_{\lambda}(\yy) \cdot s_{\mu}(\xx)$,
regarded as an element of the ring $\Lambda(\xx) \otimes \Lambda(\yy)$ 
of formal power series which are separately symmetric in the $\xx$ and $\yy$ 
variables.  

More generally, if $V$ is any finite-dimensional
 $\symm_m \times \symm_n$-module,
there exist unique integers $c_{\lambda, \mu} \geq 0$ such that
\begin{equation}
V \cong_{\symm_m \times \symm_n} 
\bigoplus_{\lambda, \mu} c_{\lambda,\mu} S^{\lambda} \otimes S^{\mu}.
\end{equation}
We then set
\begin{equation}
\Frob(V) := \sum_{\lambda, \mu} c_{\lambda,\mu} s_{\lambda}(\yy) s_{\mu}(\xx).
\end{equation}
Finally, if $V = \oplus_{d \geq 0} V_d$ is a graded $\symm_m \times \symm_n$-module
with each graded piece $V_d$ finite-dimensional, we set
\begin{equation}
\grFrob(V;q) := \sum_{d \geq 0} \Frob(V) \cdot q^d.
\end{equation}

If $U$ is a graded $\symm_m$-module and $W$ is a graded $\symm_n$-module,
we have
\begin{equation}
\grFrob(U \otimes W; q) = \grFrob(U;q) \cdot \grFrob(W; q).
\end{equation}
Recall that the $\symm_m \times \symm_n$-module $V_{n,m}$ is defined by
\begin{equation}
V_{n,m} = \bigoplus_{k \geq 0} (R_{m,k-1} \otimes R_{n,k}) \{-mn + km + kn - n - k(k-1) \}.
\end{equation}
Applying Equation~\eqref{r-frobenius}, we see that the RHS of 
Equation~\eqref{omega-to-d-equation} may be expressed as
\begin{equation}
\sum_{k \geq 0} q^{mn - km - kn + n + k(k-1)}
D_{m,k-1}(\yy;q) \cdot D_{n,k}(\xx;q) =
\grFrob(V_{n,m};q).
\end{equation}

On the other hand, for any graded $\symm_m \times \symm_n$-module $V$ and any
partition $\nu \vdash m$, we have
\begin{equation}
\text{coefficient of $s_{\nu}(\yy)$ in $\grFrob(V; q)$} =
\grFrob( \Hom_{\symm_m}(S^{\nu},V); q).
\end{equation}
Therefore, we have 
\begin{equation}
(\rev_q \circ \omega_{\xx}) \Delta'_{s_{\nu}} e_n(\xx)|_{t = 0} =
q^{-b(\nu)} \cdot \grFrob( \Hom_{\symm_m}(S^{\nu}, V_{n,m});q),
\end{equation}
which is what we wanted to prove.
\end{proof}

\section{Closing remarks}
\label{Closing}

In this paper we found an expansion of 
$\omega \Delta'_{s_{\nu}} e_n|_{t = 0}$ in the dual Hall-Littlewood basis
for any partition $\nu \vdash m$.  This led to the algebraic interpretation
of $(\rev_q \circ \omega) \Delta'_{s_{\nu}} e_n|_{t = 0}$  presented in
Theorem~\ref{algebraic-interpretation}
 involving  tensor products of $R_{n,k}$ modules.  
 It may be interesting to find a simpler module whose graded Frobenius image
 is $(\rev_q \circ \omega) \Delta'_{s_{\nu}} e_n|_{t = 0}$.

 Let $k \leq n$ be positive integers.  The ring $R_{n,k}$ has the following
 geometric interpretation.
 Denote by $\PP^{k-1}$  the $(k-1)$-dimensional complex projective space of 
 lines through the origin in $\CC^k$ and let 
 $(\PP^{k-1})^n$ denote the $n$-fold Cartesian product of $\PP^{k-1}$ with itself.
In joint work with Pawlowski \cite{PR}, the second author defined the open subvariety
$X_{n,k} \subseteq (\PP^{k-1})^n$ given by
\begin{equation}
X_{n,k} := \{ (\ell_1, \dots, \ell_n) \in (\PP^{k-1})^n \,:\, \ell_1 + \cdots + \ell_n = \CC^k \}.
\end{equation}
A typical point in $X_{n,k}$ is an $n$-tuple $(\ell_1, \dots, \ell_n)$ of one-dimensional
subspaces of $\CC^k$ which together span $\CC^k$.

The symmetric group $\symm_n$ acts on $X_{n,k}$ by the rule 
$\pi.(\ell_1, \dots, \ell_n) := (\ell_{\pi_1}, \dots, \ell_{\pi_n})$ for any permutation
$\pi = \pi_1 \dots \pi_n \in \symm_n$.
If $H^{\bullet}(X_{n,k})$ denotes the singular cohomology of $X_{n,k}$ with integer coefficients,
this gives rise to an action of $\symm_n$ on $H^{\bullet}(X_{n,k})$.  
Pawlowski and the second author prove \cite{PR} that
\begin{equation}
\label{cohomology-identification}
H^{\bullet}(X_{n,k}) = 
\ZZ[x_1, \dots, x_n]/
\langle e_n, e_{n-1}, \dots, e_{n-k+1}, x_1^k, x_2^k, \dots, x_n^k \rangle.
\end{equation}
The identification~\eqref{cohomology-identification} may be regarded as both 
an isomorphism of graded rings and an isomorphism of graded $\ZZ[\symm_n]$-modules.
The variable $x_i$ represents the Chern class $c_1(\ell_i^*)$ of the dual to the 
$i^{th}$ tautological line bundle  $\ell_i \twoheadrightarrow X_{n,k}$.
In particular, we have $R_{n,k} = \QQ \otimes_{\ZZ} H^{\bullet}(X_{n,k})$.

The isomorphism~\eqref{cohomology-identification}, together with the fact that
$X_{n,n}$ is homotopy equivalent to the manifold 
$\mathcal{F\ell}(n)$ of complete flags in $\CC^n$,
justify the statement that $X_{n,k}$ is the flag variety attached to the 
Delta Conjecture (i.e. the Macdonald eigenoperator $\Delta'_{e_{k-1}}$).
Given an arbitrary partition $\nu$, it would be interesting to find an analogous
variety $X_{n,\nu}$ with 
an action of $\symm_n$ which would play the corresponding role
for the operator $\Delta'_{s_{\nu}}$.  That is, the cohomology ring 
$H^{\bullet}(X_{n,\nu})$ should carry an action of $\symm_n$ such that
(upon tensoring with $\QQ$), the graded Frobenius image of this action
is $(\rev_q \circ \omega) \Delta'_{s_{\nu}} e_n|_{t = 0}$.
The space $X_{n,k}$ solves this problem when $\nu = (1^{k-1})$.
Theorem~\ref{algebraic-interpretation} might be helpful in constructing 
such a space $X_{n,\nu}$ in general.

We close by giving a geometric interpretation of Equation~\eqref{omega-to-c-equation},
recapitulated here:
\begin{equation*}
\sum_{\nu \vdash m} s_{\nu}(\yy) \cdot \omega_{\xx} \Delta'_{s_{\nu}} e_n(\xx) |_{t = 0} = 
 \sum_{k \geq 0} q^{m-k+1} \omega_{\yy} C_{m, k-1}(\yy; q) \cdot \omega_{\xx} C_{n, k}(\xx; q).
\end{equation*}
If $M = M_0 \oplus M_1 \oplus \cdots \oplus M_d$ is any graded vector space with 
$M_d \neq 0$, let 
$\widetilde{M} = \widetilde{M}_0 \oplus \widetilde{M}_1 \oplus \cdots \oplus \widetilde{M}_d$
 be the reversed graded vector space with components
\begin{equation}
\widetilde{M}_i := M_{d-i}, \quad 0 \leq i \leq d.
\end{equation}
In terms of reversals of $R$-modules, 
Equation~\eqref{omega-to-c-equation} reads
\begin{equation}
\omega \Delta'_{s_{\nu}} e_n |_{t = 0} =
\grFrob( \Hom_{\symm_n}(S^{\nu}, W_{n,m}) ),
\end{equation}
where 
\begin{equation}
W_{n,m} := \bigoplus_{k \geq 0} (\widetilde{R}_{m,k-1} \otimes \widetilde{R}_{n,k})
\{-m+k-1\}.
\end{equation}
Taking the reversal $\widetilde{R}_{n,k}$ of the quotient 
$R_{n,k} := \QQ[x_1, \dots, x_n]/I_{n,k}$ is not a natural ring-theoretic operation,
but it has a geometric interpretation in terms of the variety $X_{n,k}$.

Let $X_{n,k}^+ = X_{n,k} \cup \{\infty\}$ denote the one-point compactification of $X_{n,k}$,
where $\infty$ is the adjoined point.  The {\em Borel-Moore homology} 
$\bar{H}_{\bullet}(X_{n,k})$ of 
$X_{n,k}$ is the (singular) homology of the pair $(X_{n,k}^+, \{\infty\})$:
\begin{equation}
\bar{H}_{\bullet}(X_{n,k}) := H_{\bullet}(X_{n,k}^+, \{ \infty \}).
\end{equation}

The action of $\symm_n$ on $R_{n,k}$ is both continuous and proper, and so induces
a (graded) action of $\symm_n$ on the Borel-Moore homology
 $\bar{H}_{\bullet}(X_{n,k})$.  By Poincar\'e duality and 
 Equation~\eqref{cohomology-identification}, we have
 the isomorphism of $\symm_n$-modules
 \begin{equation}
 \widetilde{R}_{n,k} \cong_{\symm_n} \QQ \otimes_{\ZZ} \bar{H}_{\bullet}(X_{n,k}).
 \end{equation}
 This yields a geometric interpretation of 
 Equation~\eqref{omega-to-c-equation} in terms of Borel-Moore homology.

\section{Acknowledgements}
\label{Acknowledgements}

J. Haglund was partially supported by NSF Grant DMS-1600670.
B. Rhoades was partially supported by NSF Grant DMS-1500838.
M. Shimozono was partially supported by NSF Grant DMS-1600653.

\end{document}